\documentclass[11pt]{amsart}
\usepackage[english]{babel}
\usepackage[latin1]{inputenc}
\usepackage{amssymb,geometry,color,xcolor,graphicx}
\usepackage{tikz}\usetikzlibrary{calc}
\usepackage{amsmath,amsthm}
\usepackage{booktabs}
\usepackage{caption}
\usepackage{datetime}
\usepackage{dsfont}
\usepackage{amstext}
\usepackage{graphicx}
\usepackage{fancyhdr}
\usepackage{latexsym}
\usepackage{mathrsfs}
\usepackage{cases}
\usepackage{mathtools}
\usepackage{bm}
\usepackage{subfigure}
\usepackage{wrapfig}
\usepackage{hyperref}
\hypersetup{
    colorlinks=true, 
    linktoc=all,     
    linkcolor=blue,  
}
\usepackage{pgfplots}
\pgfplotsset{width=10cm,compat=1.9}

\numberwithin{equation}{section}

\newcommand{\udef}{\mathrel{\mathop:}=}

\newcommand{\de}{\,\mathrm{d}}
\usepackage{dsfont}

\newcommand{\norm}[1]{\left\| #1 \right\|}

\theoremstyle{plain}
\newtheorem{thm}{Theorem}[section]

\newtheorem{lm}[thm]{Lemma}

\newtheorem{rem}[thm]{Remark}
\newtheorem{exm}[thm]{Example}

\let\tilde\widetilde
\let\hat\widehat

\newcommand{\parder}[2]{\frac{\partial#1}{\partial#2}}

\renewcommand{\ss}{\scriptstyle}
\def\svdots{\vbox{\baselineskip=1.5pt\lineskiplimit=0pt
	\kern1.5pt \hbox{$\ss .$}\hbox{$\ss .$}\hbox{$\ss .$}}}

\begin{document}

\title{Convergence Analysis of a Spectral Numerical Method for a Peridynamic Formulation of Richards' Equation}

\author[Difonzo]{Fabio V. Difonzo}
\address{Istituto per le Applicazioni del Calcolo \textquotedblleft Mauro Picone\textquotedblright, Consiglio Nazionale delle Ricerche, Via G. Amendola 122/I, 70126 Bari, Italy}
\email{fabiovito.difonzo@cnr.it}

\author[Pellegrino]{Sabrina F. Pellegrino}
\address{Dipartimento di Ingegneria Elettrica e dell'Informazione, Politecnico di Bari, 70125 Bari, Italy}
\email{sabrinafrancesca.pellegrino@poliba.it}




\begin{abstract}
We study the implementation of a Chebyshev spectral method with forward Euler integrator proposed in~\cite{BDFP} to investigate a peridynamic nonlocal formulation of Richards' equation. We prove the convergence of the fully-discretization of the model showing the existence and uniqueness of a solution to the weak formulation of the method by using the compactness properties of the approximated solution and exploiting the stability of the numerical scheme. We further support our results through numerical simulations, using initial conditions with different order of smoothness, showing reliability and robustness of the theoretical findings presented in the paper.
\end{abstract}

\maketitle

\pagestyle{myheadings}
\thispagestyle{plain}
\markboth{F.V. DIFONZO AND S.F. PELLEGRINO}{ANALYSIS OF A SPECTRAL METHOD FOR A PERIDYNAMIC RICHARDS' EQUATION}

\section{Introduction}

Richards' equation is a prominent tool in the description of porous media phenomena, specifically dealing with water movement in unsaturated soils. It is derived by applying Darcy-Buckingham law to the law of mass conservation for an incompressible porous medium and constant liquid density. Existence and uniqueness of the original formulation of Richards' equation are due to \cite{AltLuckhaus1983} (see also \cite{merzRybka2010,berardiDifonzoEFM2020} and references therein). However, determining analytical solutions to Richards' equation is prohibitive under general setting on the constitutive relations typically used in the local formulation of the equation, and so numerical procedures are needed to provide explicitly computed solutions. As is well known, Richards' equation is a highly nonlinear, and possibly degenerate, parabolic equation, for which standard numerical schemes for parabolic equations fail to return reliable solutions. In fact, several approaches have been investigated according to the nature of soil through which water movement occurs: for homogeneous soils we refer to, among others, \cite{Berardi_Difonzo_Notarnicola_Vurro_APNUM_2019,Bergamaschi_Putti,Lai_Ogden_JHydrology_2015_1D_pred_corr,Feo2021}; for heterogeneous media several different approaches have been proposed, using piecewise smooth dynamical system tools (see \cite{Berardi_Difonzo_Vurro_Lopez_ADWR_2018,Berardi_Difonzo_Lopez_CAMWA_2020}); linear domain decomposition (see \cite{Arico_Sinagra_Tucciarelli,Seus_Mitra_Pop_Radu_Rohde_2018}); Kirchhoff transform (see \cite{BerardiDifonzoJCD2022,Suk_Park_J_Hydrology_2019}); finite element methods (see \cite{Manzini_Ferraris_ADWR_2004,Bachini2021}); formal asymptotics (see \cite{Kumar_List_Pop_Radu_JPC_2020}). As a general reference for the numerical features in Richards' equation, the interested reader is referred to the survey \cite{Farthing_Ogden_SSSAJ_2017}, whereas \cite{Paniconi_Putti_WRR_2015} frames Richards' equation into the context of hydrological modeling. \\
However, as common in diffusion phenomena through porous media, a nonlocal approach carries features and properties possibly useful for further analysis. This idea traces back to the '60s (see \cite{Rawlins_Gardner_1963}), and since then there has been an increasing interest, involving nonlocal behaviors in the hydraulic conductivity (see \cite{Guerrini_Swartzendruber_SSSAJ_1992}); fractional terms in the time derivative of water content (see \cite{Pachepsky_et_al_JoH_2003,Kavvas_et_al_HESS_2017}); or, also, using memory component in modeling water stress in the root water uptake (see \cite{Wu_BenGal_et_alAGWAT_2020,Carminati_VZJ_2012,Berardi_et_al_TiPM_2022,BerardiGirardiMemory2024}). \\
In the context of nonlocal formulations of Richards' equation, \cite{di2013nonlocal} extended the equation to incorporate nonlocal effects, providing a foundation for studying capillary flows. Later, in \cite{JabakhanjiMohtar2015}, the peridynamic paradigm has been applied to better describe the porous media and the dynamics of water therein, paving the way for a powerful approach to deal with the nonlinear terms in Richards' equation. \\
However, these nonlocal variants introduce challenges and opportunities, requiring specialized numerical schemes. In \cite{BDFP} authors propose an explicit Euler numerical scheme, based on Chebyshev spectral method, to solve a nonlocal formulation of Richards' equation. Therein and in \cite{difonzodilena2023}, several examples have been provided supporting the properties that the proposed numerical scheme should retain order 2 in space and order 1 in time, under mild smoothness assumptions on the initial conditions.

Spectral methods seem to be very efficient and accurate when applied to nonlocal peridynamic models. Indeed, they can benefit of the convolution-based definition of the integral operator and as a consequence they can exploit the properties of the Fast Fourier Transform (FFT) algorithm. Following this idea, in~\cite{CFLMP,Pellegrino2020} the authors perform a comparison between the implementation of Fourier spectral methods and quadrature formulas. However, trigonometric polynomials need to require periodic boundary conditions, so they cannot be applied alone to more general models. A way to overcome the issue is to make a volume penalization at the boundaries as in~\cite{LP,Jafarzadeh,LP2021} or to replace Fourier polynomials by Chebyshev polynomials, as in~\cite{LPcheby,LPcheby2022,LPeigenv}.

Spectral spatial discretization based on the approximation of the solution by means of a finite series of Chebyshev polynomials is suitable to incorporate Dirichlet boundary conditions and allows to get a high-order accuracy when applied to the nonlocal peridynamic formulation of Richards' equation (see, for instance,~\cite{BDFP}).

The convergence analysis of a specific numerical scheme tailored for the nonlocal variant is the focus of this paper, building upon state-of-the-art techniques in numerical analysis, mesh-free methods, and adaptive discretization strategies (see also~\cite{ALEBRAHIM2023109710}).

The remaining of the paper is structured as follows.
In Section~\ref{sec:nonlocRichards} we present the model, its spatial discretization and we recall the convergence result for the semi-discrete scheme. Section~\ref{sec:fully} is devoted to the deduction of the fully spectral discretization of the model and provide a rigorous proof of its convergence to a weak solution to the proposed nonlocal Richards' model. Section~\ref{sec:tests} provides some numerical simulations and finally Section~\ref{sec:conclusions} concludes the paper.


\section{A nonlocal formulation of Richards' equation based on peridynamics}
\label{sec:nonlocRichards}

We consider the following peridynamic formulation of Richards' equation with Dirichlet boundary conditions proposed in~\cite{JabakhanjiMohtar2015}
\begin{equation}\label{eq:model}
\begin{cases}
\parder{\theta}{t} (z,t)= \int_{B_\delta(z)}\frac{\varphi(z'-z)}{|z'-z|}\frac{K(z)+K(z')}{2}[H(z')-H(z)]\,\de z'+S(z), &\,\ z\in[-1,1],\,t\in[0,T] \\
\theta(z,0) = \theta^{0}(z),&\,\ z\in[-1,1],  \\
\theta(-1,t) = \theta_{0}(t),&\,\ t\in[0,T],  \\
\theta(1,t) = \theta_{Z}(t),&\,\  t\in[0,T],
\end{cases}
\end{equation}
where $\theta$ represents the {\em water content}, $K$ is the {\em hydraulic conductivity function}, $H$ is the {\em hydraulic potential}, which is related to the {\em matric head} $h_m$ by $H(z,t)=h_m(z,t)+z$, and, finally, $S$ is the {\em root uptake term}.

Let also
\begin{equation}
\label{eq:L}
\mathcal{L}\left(\theta(z,t)\right)=\int_{B_\delta(z)}\frac{\varphi(z'-z)}{|z'-z|}\frac{K(z)+K(z')}{2}[H(z')-H(z)]\,\de z'
\end{equation}
denote the peridynamic integral operator in~\eqref{eq:model}. It represents the nonlocal counterpart of the diffusivity term, as it takes into account long-range interactions between water particles (see~\cite{SILLING2000,OTERKUS2014}). The length of such interactions is parameterized by the positive scalar value $\delta$ called {\em horizon}. Due to the absence of partial spatial derivatives, the model is able to remain consistent even in presence of singularities and, therefore, it can incorporate desiccation cracks. Additionally, the function $\varphi$ is the so-called {\em influence function} and represents the convolution kernel of the model, which operates as the weight of the discrete mean value of the spatial interactions.

The behavior of this function strongly defines the profile of the solution and its dispersive effects. In particular, in~\cite{BDFP}, in order to allow the boundary conditions to be effective in the model, the authors define a {\em distributed influence function} in the following way (see Figure~\ref{fig:vphi})
\begin{equation}
\label{eq:distributedInfluenceFunction}
\varphi_\delta(z)\udef
\begin{cases}
\frac{|z|-1+\delta}{\delta}, & |z|\geq1-\delta, \\
0, & |z|<1-\delta.
\end{cases}
\end{equation}

\begin{figure}
    \centering
    \begin{tikzpicture}[
declare function={
    func(\x)= and(\x > -1+0.15, \x<1-0.15) * (0)
    +
    or(\x <= -1+0.15, \x >= 1-0.15) * ((abs(\x)-1+0.15)/0.15) ;
  }
]
\begin{axis}[
    width = .55\linewidth,
    axis lines = left,
    xlabel = \(z\),
    ylabel = {\(\varphi_\delta(z)\)},
]

\addplot [
    domain=-1:1,
    samples=200,
    color=black,
] {func(x)};
\end{axis}
\end{tikzpicture}
    \caption{The distributed influence function $\varphi_{\delta}(z)$.}
    \label{fig:vphi}
\end{figure}
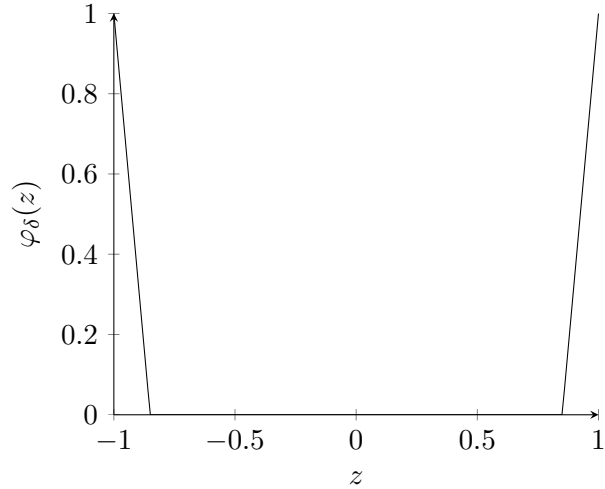

Due to the nonlinearity of the model, a numerical approach is needed in order to study the properties of the solution. In particular, in~\cite{BDFP} the model is discretized by using Chebyshev spectral collocation scheme for spatial discretization with forward Euler method for the time marching. Moreover, the authors prove the convergence of the semi-discrete method by projecting the approximated solution into the space of Chebyshev polynomials and exploiting the Lipschitz continuity of the peridynamic operator $\mathcal{L}$ in~\eqref{eq:L}.

Additionally, the authors show numerically the convergence of the fully-discrete scheme without providing a rigorous proof. The aim of this work is to complete the analysis adding the proof of the convergence of the fully-discrete scheme showing the compactness and stability properties of the approximated solution.

In what follows, we recall the construction of the spectral method for the spatial discretization and its  convergence. We refer the reader to~\cite{BDFP} for more details. Moreover, we provide a brief review of the functional spaces and of the projection operator we will use in the next section to prove the convergence of the fully-discrete method.

Let $N>0$, and $z_h\udef\cos(h\pi/N)$, for $h=0,\dots,N$ be a partition of the spatial domain $[-1,1]$ obtained by using the non-uniform Chebyshev-Gauss-Lobatto (CGL) collocation points.
\begin{rem}\label{rem:affineTransformation}
    The choice to take $[-1,1]$ as spatial domain is to simplify the computations; however, more general intervals can be considered by applying an affine transformation.
\end{rem}
We look for an approximation of the solution to~\eqref{eq:model} in the following form
\begin{equation}
\label{eq:Chebyexpansion}
\theta^N(z,t)=\sum_{k=0}^N \bar{\theta}_k(t) T_k(z),
\end{equation}
where $T_k(z)$ is the $k$-th Chebyshev polynomial of the first kind, defined as $T_k(z)\udef\cos(k\arccos{z})$, which is an orthogonal polynomial with respect to the weight $w(z)=\left(\sqrt{1-z^2}\right)^{-1}$, and $\bar{\theta}_k(t)$ is the $k$-th discrete Chebyshev coefficient given by
\begin{equation}
\label{eq:Chebycoeff}
\bar{\theta}_k(t)\udef\frac{1}{\gamma_k} \sum_{h=0}^N \theta(z_h,t) T_k(z_h) w_h,
\end{equation}
where
\begin{equation}
\label{eq:gamma}
\gamma_k\udef
\begin{cases}
    \pi,\quad&k=0,\,N,\\
    \frac{\pi}{2},\quad &k=1,\dots,N-1,
\end{cases}
\end{equation}
and
\begin{equation}
\label{eq:w}
w_h\udef
\begin{cases}
    \frac{\pi}{2N},\quad &h=0,\,N,\\
    \frac{\pi}{N},\quad &h=1,\dots,N-1.
\end{cases}
\end{equation}

We set
\begin{align*}
\Lambda(z) &\udef K(z)H(z), \\
\overline{\varphi}_\delta(z) &\udef \frac{\varphi_\delta(z)}{\norm{z}},
\end{align*}
and
\[
\beta = \int_{-1}^{1}\overline{\varphi}_\delta(z)\,\de z=2\left(1+\frac{1-\delta}{\delta}\ln(1-\delta)\right).
\]

If we replace $\theta$ by $\theta^N$ into equation~\eqref{eq:model}, thanks to the Convolution Theorem, we obtain the semi-discretization of the model at each collocation point $z_h$ as follows
\begin{equation}
\label{eq:semischeme}
\begin{split}
\parder{\theta^N}{t}(z_h,t)&=\frac12\left(\mathcal{F}^{-1}\left(\mathcal{F}\left(\overline{\varphi}_\delta\right)\mathcal{F}\left(\Lambda\right)\right)(z_h)+K(z_h)\ \mathcal{F}^{-1}\left(\mathcal{F}\left(\overline{\varphi}_\delta\right)\mathcal{F}\left(H\right)\right)(z_h)\right)\\
&\quad-\frac12\left(H(z_h)\ \mathcal{F}^{-1}\left(\mathcal{F}\left(\overline{\varphi}_\delta\right)\mathcal{F}\left(K\right)\right)(z_h)+\beta\Lambda(z_h)\right)+ S(z_h),
\end{split}
\end{equation}
with initial condition
\begin{equation}
\label{eq:initcondsemischeme}
\theta^N(z_h,0)=\theta^{0,N}(z_h),\qquad h=0,\dots,N,
\end{equation}
and boundary conditions
\begin{equation}
\label{eq:BCsemischeme}
\begin{split}
    \theta^N(z_0,t)&=\theta^N_0(t),\qquad t\in[0,T],\\
    \theta^N(z_N,t)&=\theta^N_Z(t),\qquad t\in[0,T],
\end{split}
\end{equation}
where $\mathcal{F}$ and $\mathcal{F}^{-1}$ denote the discrete Chebyshev transform and the discrete inverse Chebyshev transform defined in~\eqref{eq:Chebycoeff} and~\eqref{eq:Chebyexpansion}, respectively.

In~\cite{BDFP}, the authors prove the convergence of the semi-discrete scheme~\eqref{eq:semischeme}-\eqref{eq:initcondsemischeme}-\eqref{eq:BCsemischeme} in the space of all continuous functions in the weighted Sobolev space $H^s_w\left([-1,1]\right)$, with $w(z)=\left(\sqrt{1-z^2}\right)^{-1}$ and for any $s\ge 1$. The proof makes use of the projector operator into the orthogonal space of Chebyshev polynomials and exploits the Lipschitz boundedness of $H$ and $K$.

We introduce the space of Chebyshev polynomials of degree $N$, defined as
\[
S_{N}= \text{span}\left\{T_k(z)\ |\ 0 \le z\le N\right\}\subset L^2_w([-1,1]),
\]
and the orthogonal projection operator $P_N: L^2_w([-1,1]) \to S_N$ given by
\[
P_Nu(z)= \sum_{k=0}^N \bar{u}_k T_k(x)w_k,
\]
where the weight $w_k$ is defined in~\eqref{eq:w} and is such that for any $u\in L^2_w([-1,1])$, the following equality holds
\begin{equation}
\label{eq:orthogonal}
(u-P_Nu,\varphi)_w = \int_{-1}^1 \left(u-P_Nu\right)\varphi \ w \de z =0,\quad\text{for every $\varphi\in S_N$}.
\end{equation}


Then, using~\eqref{eq:L}, the semi-discrete scheme for~\eqref{eq:semischeme}-\eqref{eq:initcondsemischeme}-\eqref{eq:BCsemischeme} can be reformulated in terms of $P_N$ as follows
\begin{gather}
\label{eq:schemePN}
\parder{\theta^N}{t} (z,t) = P_N \mathcal{L}\left(\theta^N(z,t)\right) + P_N S(z),\\
\label{eq:initial_scheme}
\theta^N(z,0) = P_N \theta^0(z),
\end{gather}
with boundary conditions
\begin{equation}
\label{eq:BC-PN}
\begin{split}
\theta^N(-1,t)&= P_N \theta_0(t),\qquad t\in[0,T]\\
\theta^N(1,t) &= P_N \theta_Z(t),\qquad t\in[0,T],
\end{split}
\end{equation}
where $\theta^{N}(z,t)\in S_N$ for every $0\le t\le T$.

We fix $s\geq1$ and define by $X_s \udef \mathcal{C}^0\left(0,T; H^s_w\left([-1,1]\right)\right)$ the space of all continuous functions in the weighted Sobolev space $H^s_w\left([-1,1]\right)$,
with norm
\[
\norm{u}_{X_s}^2 = \max_{t\in[0,T]}\norm{u(\cdot,t)}_{s,w}^2,
\]

for any $T> 0$.

From now on, we denote by $C$ a generic positive constant. There hold the following results.

\begin{lm}[{\cite[Theorem 3.1]{Canuto}}]
\label{lm:sobolev}
For any real $0\le \mu\le s$, there exists a positive constant $C$ such that
\begin{equation}
\label{eq:sobolev}
\norm{\theta-P_N\theta}_{H^\mu_{w}([-1,1])} \le \frac{C}{N^{s-\mu}}\norm{\theta}_{H^s_{w}([-1,1])}, \quad\text{for every $\theta\in H^{s}_w([-1,1])$}.
\end{equation}
\end{lm}

\begin{thm}[{\cite[Theorem 4]{BDFP}}]
\label{th:convergence-semi}
Let $s\ge 1$ and $\theta(z,t)\in X_s$ be the solution to the initial-boundary-valued problem~\eqref{eq:model} and $\theta^N(z,t)$ be the solution to the semi-discrete scheme~\eqref{eq:schemePN}-\eqref{eq:initial_scheme}-\eqref{eq:BC-PN}.
Then, there exists a positive constant $C$, independent on $N$, such that
\begin{equation}
\label{eq:order_conv}
\norm{\theta-\theta^N}_{X_1} \le C(T) \left(\frac{1}{N}\right)^{s-1} \norm{\theta}_{X_s},
\end{equation}
for any initial data $\theta^0\in H^s_w([-1,1])$ and for any $T > 0$.
\end{thm}

\section{Fully spectral discretization of the model}
\label{sec:fully}

Let $N_T>0$ be a positive integer and $0=t_0<t_1<\dots<t_{N_T}=T$ be a uniform partition of $[0,T]$, namely, if we set $\Delta t=T/N_T$, then $t_n=n\Delta t$, for $n=0,1,\dots,N_T$. Given an arbitrary function $\psi(t)$, we write $\psi_n$ as the value of $\psi$ at $t=n\Delta t$. The backward difference form is $d_t \psi_n=\left(\psi_n-\psi_{n-1}\right)/\Delta t$ for any sequence $\{\theta_n\}$. 

We assume that $S\in L^2_w\left([-1,1]\right)$ and that the initial condition $\theta^N_0\in H^1_{w}\left([-1,1]\right)$ is such that
\begin{equation}
\label{eq:ICassumption}
\norm{\theta^0 - \theta_0^N}_{L^2_w\left([-1,1]\right)} \le \frac{C}{N^{2-\mu}}\norm{\theta^0}_{L^2_w\left([-1,1]\right)},\quad\text{for any $0\le\mu\le 2$}.
\end{equation}
Thus, the fully-discrete spectral scheme for the model can be written as
\begin{equation}
\label{eq:fully}
\begin{cases}
\theta^N_{n}=\theta^N_{n-1} + \Delta t \left(P_N\mathcal{L}\left(\theta^N_n\right) + P_N S\right), \\
\theta^N_0=P_N \theta^0_0.
\end{cases}
\end{equation}

In this section, we prove the existence and uniqueness of the solution to~\eqref{eq:fully} and that such solution converges to the solution of the continuous model~\eqref{eq:model} as $\Delta t \to 0$ and $N\to\infty$. To do so, we prove some preliminary Lemmas.

The following result provides a nonlocal counterpart of the maximum principle for strong solution of parabolic equations.

\begin{lm}[see~\cite{maxprinc2019}]\label{lm:maxprinc}
 Let $\theta$ be a strong solution to~\eqref{eq:model} for $t\in[0,T]$. Then
 \begin{equation}
\label{eq:maxprinc}
\theta(z,t) \le e^{t/2}\norm{S}_{L^2_w\left([-1,1]\right)} + \max\left\{\sup_{z\in(-1,1)} \theta^0,\sup_{t\in(0,T]} \theta_0(t),\sup_{t\in(0,T]} \theta_Z(t)\right\},
 \end{equation}
 for any $z\in[-1,1]$ and $t\in[0,T]$.
\end{lm}
As a consequence of Lemma~\ref{lm:maxprinc} we can assume that the water content $\theta$ in~\eqref{eq:model} is uniformly bounded.

\begin{lm}
\label{lm:LinL2}
Let $\theta_m^N(z)$ be the solution to the fully-discrete scheme~\eqref{eq:fully}. Then, $\mathcal{L}(\theta_m^N)\in L^2_w([-1,1])$.
\end{lm}
\begin{proof}
Due to the definition of $\varphi_\delta$ in~\eqref{eq:distributedInfluenceFunction} and since $H$ and $K$ are
locally Lipschitz, using the Cauchy-Schwartz inequality, we find
\begin{align*}
\int_{-1}^1\left(\mathcal{L}(\theta_m^N)\right)^2\,\de z &= \int_{B_1(z)}\frac{\left(\varphi_\delta(z'-z)\right)^2}{\|z'-z\|^2}\frac{\left(K(z)+K(z')\right)^2}{4}\left(H(z')-H(z)\right)^2\,\de V_{z'}<\infty,
\end{align*}
and this proves the claim.
\end{proof}

We prove the following stability property.
\begin{lm}
\label{lm:stability}
Let $\theta^N_m$ be the numerical solution of~\eqref{eq:fully} for every $1\le m\le N_T$, then $\theta^N_m$ satisfies the following stability estimate
\begin{equation}
\label{eq:stability}
\sum_{n=1}^m\norm{\theta^N_{n}-\theta^N_{n-1}}_{L^2_w\left([-1,1]\right)}^2+\norm{\theta^N_m}_{L^2_w\left([-1,1]\right)}^2 + \Delta t \sum_{n=1}^m \norm{\mathcal{L}\left(\theta^N_n\right)}_{L^2_w\left([-1,1]\right)}^2 \le C_0,
\end{equation}
where $C_0$ is a generic positive constant depending on $\theta^0$ and $S$.
\end{lm}

\begin{proof}
Let $\varphi^N_{n} = 2 \theta^N_n$. We consider the inner product with $\varphi^N_n$ in~\eqref{eq:fully}:
\begin{equation}
\label{eq:weakfully}
\frac{2}{\Delta t}\left(\theta^N_n-\theta^N_{n-1},\theta^N_n\right) = 2\left(P_N\mathcal{L}\left(\theta^N_n\right),\theta^N_n\right) + 2\left(P_NS,\theta^N_n\right).
\end{equation}
Since $2\left(a-b,a\right)=a^2-b^2+\left(a-b\right)^2$, we have
\begin{align*}
    \norm{\theta^N_n}_{L^2_w\left([-1,1]\right)}^2-\norm{\theta^N_{n-1}}_{L^2_w\left([-1,1]\right)}^2&+\norm{\theta^N_n-\theta^N_{n-1}}_{L^2_w\left([-1,1]\right)}^2\\
    &\quad =2\Delta t\left(P_N\mathcal{L}\left(\theta^N_n\right),\theta^N_n\right) +2\Delta t \left(P_NS,\theta^N_n\right).
\end{align*}
Adding over $n=1\ldots,m$, and using Cauchy inequality, Lemma~\ref{lm:LinL2} and Lemma~\ref{lm:sobolev}, we find
\begin{equation}
\label{eq:sumestimate}
\begin{split}
\norm{\theta^N_m}_{L^2_w\left([-1,1]\right)}^2 &+ \sum_{n=1}^m \norm{\theta^N_n-\theta^N_{n-1}}_{L^2_w\left([-1,1]\right)}^2 \\
&= \norm{\theta^N_0}_{L^2_w\left([-1,1]\right)} +
2\Delta t\sum_{n=1}^m\left(P_N\mathcal{L}\left(\theta^N_n\right),\theta^N_n\right) + 2\Delta t\sum_{n=1}^m\left(P_NS,\theta^N_n\right)\\
&\le \norm{\theta^N_0}_{L^2_w\left([-1,1]\right)} + 2\Delta t \norm{P_N S}_{L^2_w\left([-1,1]\right)}^2 \sum_{n=1}^m \norm{\theta^N_n}_{L^2_w\left([-1,1]\right)}^2\\
&\quad +2\Delta t \sum_{n=1}^m \norm{P_N\mathcal{L}\left(\theta^N_n\right)-\mathcal{L}\left(\theta^N_n\right)}_{L^2_w\left([-1,1]\right)}^2\norm{\theta^N_n}_{L^2_w\left([-1,1]\right)}^2\\
&\quad +2\Delta t \sum_{n=1}^m \norm{\mathcal{L}\left(\theta^N_n\right)}_{L^2_w\left([-1,1]\right)}^2 \norm{\theta^N_n}_{L^2_w\left([-1,1]\right)}^2\\
&\le \norm{\theta^N_0}_{L^2_w\left([-1,1]\right)}^2+ 2\Delta t\left(\frac{C}{N}+1\right)\le C_0,
\end{split}
\end{equation}
that proves the claim.
\end{proof}

\begin{lm}
    If $\theta^N_n$ satisfies the stability condition of Lemma~\ref{lm:stability}, then it is the unique solution to the weak formulation~\eqref{eq:fully}.
\end{lm}

\begin{proof}
    For any $\varphi^N\in S_N$, considering the inner product with $\varphi^N$ in~\eqref{eq:fully}, we have
    \begin{equation}
    \label{eq:weakrewritten}
    \frac{1}{\Delta t} \left(\theta^N_n,\varphi^N \right) = \left(P_N \mathcal{L}\left(\theta^N_n\right),\varphi^N\right) + \left(P_N S,\varphi^N\right) + \frac{1}{\Delta t}\left(\theta^N_{n-1},\varphi^N\right).
    \end{equation}
    Let us define the bilinear form
    \begin{equation}
    \label{eq:G}
        G\left(\theta^N_n,\varphi^N\right) \udef \frac{1}{\Delta t} \left(\theta^N_n,\varphi^N\right) -\left(P_N \mathcal{L}\left(\theta^N_n\right),\varphi^N\right).
    \end{equation}
    It is continuous and coercive thanks to the orthogonality of $P_N$ and Lemma~\ref{lm:LinL2}. Therefore, the solution attained for problem~\eqref{eq:weakrewritten} is unique.
\end{proof}

We introduce now some interpolated functions. Let $\theta^N_{\Delta t}(\cdot,t)$ be the piecewise linear continuous interpolation of the solution $\theta^N_n$, $n=1,\dots,N$ on the time interval $(t_{n-1},t_n]$, namely
\begin{equation}
\label{eq:interpolatedtheta}
\theta^N_n (\cdot,t) = \frac{t-t_{n-1}}{\Delta t}\ \theta^N_{n}(\cdot,t) + \frac{t_n-t}{\Delta t}\ \theta^N_{n-1}.
\end{equation}
Moreover, we define the piecewise constant extensions of $\theta^N_n$ and $\theta^N_{n-1}$ respectively as follows
\begin{equation}
\label{eq:constantextensiontheta}
\begin{split}
\tilde{\theta}^N_{\Delta t} (\cdot,t) &= \theta^N_n,\\
\hat{\theta}^N_{\Delta t} (\cdot, t) &= \theta^N_{n-1},
\end{split}
\end{equation}
for any $t\in (t_{n-1},t_n]$.

The next result is an a-priori stability estimate on $\theta^N_{\Delta t}$, independent on $N$ and $\Delta t$.
\begin{lm}
\label{lm:Hminus1}
    Given the sequence $\{\theta^N_{\Delta t}\}$, there exists a positive constant $C>0$ independent on $N$ and $\Delta t$ such that
    \begin{equation}
    \label{eq:dualestimate}
    \norm{\partial_t \theta^N_{\Delta t}}_{L^2\left(0,T;L^2_w\left([-1,1]\right)\right)}\le C.
    \end{equation}
\end{lm}
\begin{proof}
    Cauchy inequality gives us
    \begin{align*}
        \int_0^T \int_{-1}^1 \left|\partial_t \theta^N_{\Delta t} \varphi^N\right|\de z \de t&=\int_0^T\int_{-1}^1\left|P_N\mathcal{L}\left(\theta^N_{\Delta t}\right) + P_N S\right|\left|\varphi^N\right|\de z \de t\\
        &\quad\le \frac{1}{2}\int_0^T\int{-1}^1 \left(P_N\mathcal{L}\left(\theta^N_{\Delta t}\right)\right)^2\left(\varphi^N\right)^2 \de z \de t \\
        &\quad + \frac{1}{2} \int_0^T \int_{-1}^1 \left(P_N S\right)^2 \left(\varphi^N\right)^2.
    \end{align*}
The claim is proved.
\end{proof}

Now we can prove the convergence result for the fully-discrete solution.
\begin{thm}
   There exists a function $\theta\in L^2\left(0,T;L^2_w\left([-1,1]\right)\right)$ such that, as $N\to \infty$ and $\Delta t\to 0$, there hold
   \begin{equation}
    \label{eq:convthm}
    \begin{split}
    \tilde{\theta}^N_{\Delta t},\,\hat{\theta}^N_{\Delta t},\,\theta^N_{\Delta t}\rightharpoonup\theta \qquad\text{weakly in}\quad &L^2\left(0,T;{L^2_w\left([-1,1]\right)}\right),\\
    \partial_t \theta^N_{\Delta t} \rightharpoonup \partial_t\theta\qquad\text{weakly in}\quad&L^2\left(0,T;L^2_w\left([-1,1]\right)\right),\\
    \tilde{\theta}^N_{\Delta t},\,\hat{\theta}^N_{\Delta t},\,\theta^N_{\Delta t}\to\theta \qquad\quad \qquad \text{in}\quad&L^2\left(0,T;L^q_w\left([-1,1]\right)\right),
    \end{split}
   \end{equation}
   with $1\le q\le 2$.
\end{thm}
\begin{proof}
    Lemma~\ref{lm:stability} ensures that the sequences $\{\tilde{\theta}^N_{\Delta t}\}$, $\{\hat{\theta}^N_{\Delta t}\}$ and $\{\theta^N_{\Delta t}\}$ are bounded and, as a consequence, each of them admits a weak convergent subsequence.

    We prove now that these sequences (still denoted by the same way to lighten the notation) converge to the same limit $\theta$. Indeed, using the interpolation inequality, Cauchy inequality and Lemma~\ref{lm:stability} we obtain
    \begin{align*}
        \norm{\theta^N_{\Delta t}-\tilde{\theta}^N_{\Delta t}}_{L^2\left(0,T;L^q_w\left([-1,1]\right)\right)}^2 &\le \Delta t \sum_{n=1}^m \norm{\theta^N_{n}-\theta^N_{n-1}}_{L^q_w\left([-1,1]\right)}^2\\
        &\le \Delta t \sum_{n=1}^m \norm{\theta^N_{n}-\theta^N_{n-1}}_{L^1_w\left([-1,1]\right)}^{2\alpha} \norm{\theta^N_{n}-\theta^N_{n-1}}_{L^2_w\left([-1,1]\right)}^{2-2\alpha}\\
        &\le C \left(\Delta t\right)^{\alpha}\left(\sum_{n=1}^m \norm{\theta^N_{n}-\theta^N_{n-1}}_{L^2_w\left([-1,1]\right)}^{2} \right)^{\alpha} \\
        &\quad \left(\Delta t\sum_{n=1}^m \norm{\theta^N_{n}-\theta^N_{n-1}}_{L^2_w\left([-1,1]\right)}^{2}\right)^{1-\alpha} \\
        &\stackrel{\Delta t \to 0}{\longrightarrow} 0,
    \end{align*}
    where $\alpha=\frac{2-q}{q}$. Similarly, we find
    \begin{equation*}
        \norm{\theta^N_{\Delta t}-\hat{\theta}^N_{\Delta t}}_{L^2\left(0,T;L^q_w\left([-1,1]\right)\right)}^2 \stackrel{\Delta t \to 0}{\longrightarrow} 0.
    \end{equation*}

    Finally, these convergences are strong in $L^2\left(0,T;L^q_w\left([-1,1]\right)\right)$ thanks to Aubin-Lions Lemma and Lemma~\ref{lm:Hminus1}.
\end{proof}

\section{Numerical Simulations}
\label{sec:tests}

In this section we test our proposed method on different soils with different initial conditions: in Example \ref{ex:1} we use a function with a discontinuity in its first derivative; in Example \ref{ex:2} we use a periodic function. Moreover, in both cases a sink forcing term $S(z)$ is active as in \eqref{eq:model}, representing the water uptake due to root systems.

Also, we consider the classical Van Genuchten-Mualem constitutive relations in the unsaturated zone, given by
\begin{align*}
\theta\left( h_m \right) &= \theta_r + \frac{\theta_S - \theta_r}{\left( 1 + |\alpha h_m|^n \right)^m}, \quad m \udef 1- \frac{1}{n}, \\
K(h_m) &= K_S  \left[ \frac{1}{1 + |\alpha h_m|^n} \right]^{\frac{m}{2}} \left[ 1 - \left(1 - \frac{1}{1 + |\alpha h_m|^n} \right)^m\right]^2,
\end{align*}
where $\theta_r$ and $\theta_S$ represent the residual and the saturated water content, respectively, $K_S$ the saturated hydraulic conductivity, and $\alpha,\, n$
are fitting parameters. Moreover, according to Remark~\ref{rem:affineTransformation}, we perform our simulations in the spatial domain $[0,Z]$.
\begin{exm}\label{ex:1}
\begin{figure}
    \centering
    \includegraphics[width=0.7\textwidth]{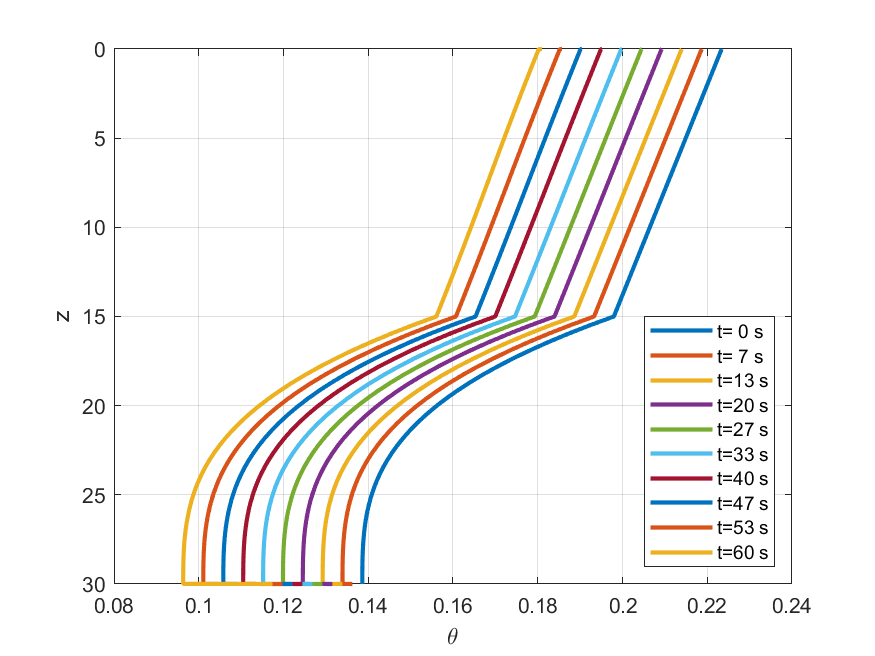}
    \caption{With reference to Example~\ref{ex:1}, the profile of the soil moisture for different time values. The parameters of the simulations are $N=100$, $\Delta t=0.06$ s and $\delta=0.15$.}
    \label{fig:ex1}
\end{figure}

As in~\cite{Berardi_Difonzo_Notarnicola_Vurro_APNUM_2019,Haverkamp}, we consider a sand with parameters
\[
\theta_r=0.075,\,\theta_S=0.287,\,\alpha=0.036,\,n=1.56,\,K_S=0.00094\,\textrm{cm/s}.
\]
We added a sink term $S=-700\,\textrm{s}^{-1}$ and parameter $\delta=0.15$ in \eqref{eq:distributedInfluenceFunction}. We set our initial and boundary conditions as follows
\begin{align*}
\theta(0,t) &= 0.2234\left(1-\frac{t}{T}\right)+0.1810\frac{t}{T},\,\,t\in[0,T], \\
\theta(Z,t) &= 0.1386\left(1-\frac{t}{T}\right)+0.1174\frac{t}{T},\,\,t\in[0,T],
\end{align*}
while initial condition is defined as
\[
\theta(z,0)=
\begin{cases}
0.1386+0.0594(x+1), & x\in[-1,0], \\
0.2234+0.0254(x-1), & x\in[0,1],
\end{cases}
\quad x\udef\frac{Z-2z}{Z},\,\,z\in[0,Z],
\]
showing a discontinuity in the first derivative at $z=\frac{Z}{2}$. \\
We select $Z=30$ cm, $T=60$ s; moreover, we used $\Delta t=0.06$ s and $N=100$. Results are shown in Figure \ref{fig:ex1}.
\end{exm}

\begin{exm}\label{ex:2}
\begin{figure}
    \centering
    \includegraphics[width=0.7\textwidth]{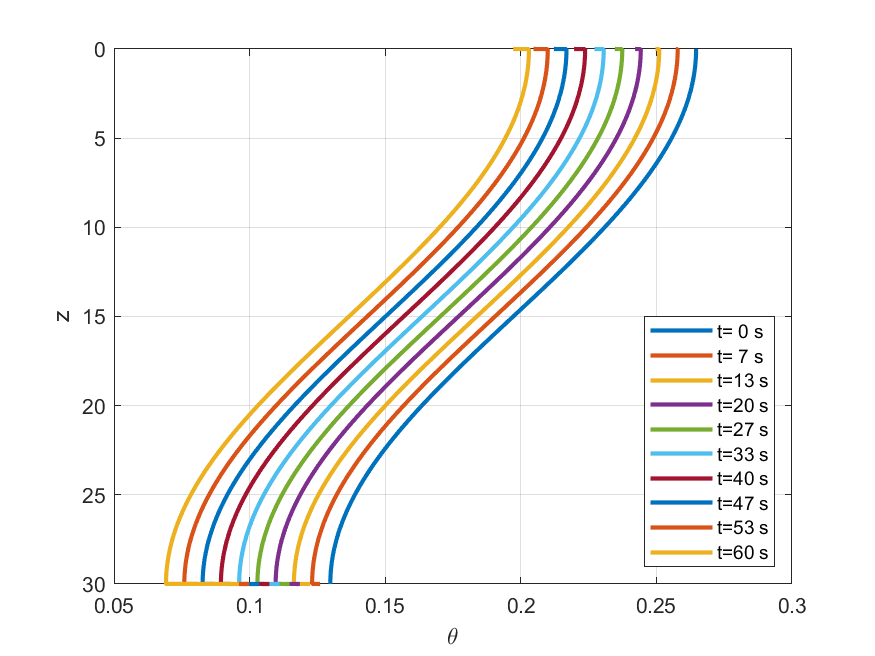}
    \caption{With reference to Example~\ref{ex:2}, the profile of the soil moisture for different time values. The parameters of the simulations are $N=100$, $\Delta t=0.06$ s and $\delta=0.15$.}
    \label{fig:ex2}
\end{figure}

As in \cite{Hills_et_al_1989}, we consider a Hills Berino loamy fine sand with parameters
\[
\theta_r=0.0286,\,\theta_S=0.3658,\,\alpha=0.028,\,n=2.2390,\,K_S=0.0063\,\textrm{cm/s}.
\]
We added a sink term $S=-1000\,\textrm{s}^{-1}$ and parameter $\delta=0.15$ in \eqref{eq:distributedInfluenceFunction}. We set our initial and boundary conditions as follows
\begin{align*}
\theta(0,t) &= 0.2646\left(1-\frac{t}{T}\right)+0.1972\frac{t}{T},\,\,t\in[0,T], \\
\theta(Z,t) &= 0.1298\left(1-\frac{t}{T}\right)+0.0960\frac{t}{T},\,\,t\in[0,T],
\end{align*}
while initial condition is defined as the periodic function
\[
\theta(z,0)=-0.0674\cos\left(\frac{x+1}{2}\pi\right)+0.1972,\,\,x\udef\frac{Z-2z}{Z},\,\,z\in[0,Z],
\]
We select $Z=30$ cm, $T=60$ s; moreover, we used $\Delta t=0.06$ s and $N=100$. Results are shown in Figure \ref{fig:ex2}.
\end{exm}

\section{Conclusions}
\label{sec:conclusions}

We have studied a fully-discrete spectral scheme for a nonlocal formulation of Richards' equation based on the peridynamic theory. We prove the convergence of the method to the unique weak solution to the problem as the timestep size tends to zero and the total number of collocation points used for the discretization of the spatial domain goes to infinity. The proof is based on the fact that the numerical approximation of the solution satisfies the stability and the compactness properties. Finally, we have given some simulations to show a numerical verification of the existence of weak solution to our model.

The present work suggests several possible directions for future and already ongoing research studies. In particular, it would be of interest study the convergence of the scheme when we reduce the regularity of the initial conditions to a Radon measure (see for instance~\cite{LIMICROGINA2010}).
Moreover, we plan to construct a generalization of the model to 2D in order to represent and to study the evolution of desiccation cracks implicitly incorporated into the model. To do this, in order to avoid the Gibbs' phenomenon near discontinuities, we would investigate the implementation of a filtering strategy coupled with the Chebyshev spectral discretization as in~\cite{Pellegrino2023}.

\section*{Acknowledgments}

FVD has been supported by \textit{REFIN} Project, grant number 812E4967 funded by Regione Puglia.

SFP has been supported by \textit{REFIN} Project, grant number D1AB726C funded by Regione Puglia, and by \textit{PNRR MUR - M4C2} project, grant number N00000013 - CUP D93C22000430001.

The two authors gratefully acknowledge the support of INdAM-GNCS 2023 Project, grant number CUP$\_$E53C22001930001.

\bibliographystyle{plain}
\bibliography{biblio.bib}

\end{document}